\documentclass{article}%
\usepackage{graphicx}
\usepackage[intlimits]{amsmath}
\usepackage{latexsym}
\usepackage{amsfonts}
\usepackage{amssymb}%
\makeatletter
\newfont{\footsc}{cmcsc10 at 8truept}
\newfont{\footbf}{cmbx10 at 8truept}
\newfont{\footrm}{cmr10 at 10truept}
\newtheorem{theorem}{Theorem}

\newtheorem{conjecture}[theorem]{Conjecture}

\newtheorem{lemma}[theorem]{Lemma}

\newenvironment{proof}[1][Proof]{\noindent{\textbf {#1}  }}  {\hfill$\Box$\bigskip}

\begin{document}

\title{Maxima of the Q-index: forbidden odd cycles }
\author{Xiying Yuan\thanks{Department of Mathematics, Shanghai University, Shanghai,
200444, China; email: \textit{xiyingyuan2007@hotmail.com}} \thanks{Research
supported by National Science Foundation of China (No. 11101263), and by a
grant of \textquotedblleft The First-class Discipline of Universities in
Shanghai\textquotedblright.} }
\maketitle

\begin{abstract}
Let $q\left(  G\right)  $ be the $Q$-index (the largest eigenvalue of the
signless Laplacian) of $G$. Let $S_{n,k}$ be the graph obtained by joining
each vertex of a complete graph of order $k$ to each vertex of an independent
set of order $n-k.$ The main result of this paper is the following theorem:
$\medskip$

Let $k\geq3,$ $n\geq110k^{2},$ and $G$ be a graph of order $n$. If $G$ has no
$C_{2k+1},$ then $q\left(  G\right)  <q\left(  S_{n,k}\right)  ,$ unless
$G=S_{n,k}.$

This result proves the odd case of the conjecture in [M.A.A. de Freitas, V.
Nikiforov, and L. Patuzzi, Maxima of the $Q$-index: forbidden $4$-cycle and
$5$-cycle, \emph{Electron. J. Linear Algebra }26 (2013), 905-916.]



\end{abstract}

\section{Introduction}

Let $G=\left(  V\left(  G\right)  ,E\left(  G\right)  \right)  $ be a simple
graph. Denote by $\nu\left(  G\right)  $ the order of $G,$ and $e\left(
G\right)  $ the size of $G,$ that is to say, $\nu\left(  G\right)  =|V\left(
G\right)  |,$ and $e\left(  G\right)  =|E\left(  G\right)  |.$ Set $\Gamma
_{G}\left(  u\right)  =\left\{  v\text{
$\vert$
}uv\in E\left(  G\right)  \right\}  ,$ and $d_{G}\left(  u\right)
=|\Gamma_{G}\left(  u\right)  |;$ or simply $\Gamma\left(  u\right)  $ and
$d\left(  u\right)  ,$ respectively. Let $\delta\left(  G\right)  $ and
$\Delta\left(  G\right)  $ denote the minimal degree and maximal degree of
graph $G,$ respectively. As usual, we may let $P_{k},$ $C_{k},$ and $K_{k}$ be
the path, cycle, and complete graph of order $k,$ respectively$.$ Let
$S_{n,k}$ be the graph obtained by joining each vertex of a $K_{k}$ to each
vertex of a $\overline{K}_{n-k};$ in other words, $S_{n,k}=K_{k}\vee
\overline{K}_{n-k}.$ Also, let $S_{n,k}^{+}$ be the graph obtained by adding
an edge to $S_{n,k}.$ For notation and concepts undefined here, we refer the
reader to \cite{Bol98}.

The $Q$-index of $G$ is the largest eigenvalue $q\left(  G\right)  $ of its
signless Laplacian $Q\left(  G\right)  $. For more basic facts about $q\left(
G\right)  ,$ we refer the reader to \cite{C10}. The following conjecture has
been proposed in \cite{FNP13}.

\begin{conjecture}
\label{con1} Let $k\geq2$ and let $G$ be a graph of sufficiently large order
$n.$ If $G$ has no $C_{2k+1},$ then $q\left(  G\right)  <q\left(
S_{n,k}\right)  ,$ unless $G=S_{n,k}.$ If $G$ has no $C_{2k+2},$ then
$q\left(  G\right)  <q\left(  S_{n,k}^{+}\right)  ,$ unless $G=S_{n,k}^{+}.$
\end{conjecture}

In \cite{Nik201310}, Conjecture \ref{con1} was solved asymptotically by the
following results.

\begin{theorem}
\cite{Nik201310} Let $k\geq2,$ $n>6k^{2},$ and let $G$ be a graph of order
$n.$ If $q\left(  G\right)  \geq n+2k-2,$ then $G$ contains a cycle of length
$l$ for each $l\in\left\{  3,4,\ldots,2k+2\right\}  .$
\end{theorem}

There are many helpful guidance provided in \cite{Nik201310} for solving
Conjecture \ref{con1}. By following the clue provided in \cite{Nik201310} and
by some analysis we will give the complete solution for the odd case of
Conjecture \ref{con1}. More precisely, we will prove the following result.

\begin{theorem}
\label{main result} Let $k\geq3,$ $n\geq110k^{2},$ and let $G$ be a graph of
order $n$.\ If $G$ has no $C_{2k+1},$ then $q\left(  G\right)  <q\left(
S_{n,k}\right)  ,$ unless $G=S_{n,k}.$
\end{theorem}

\section{Some auxiliary results}

First we state some well-known results, which will be used in the following proofs.

\begin{lemma}
\label{EGp} \cite{ErGa59}Let $k\geq1.$ If $G$ is a graph of order $n,$ with no
$P_{k+2},$ then $e\left(  G\right)  \leq kn/2,$ with equality holding if and
only if $G$ is a union of disjoint copies of $K_{k+1}.$
\end{lemma}

\begin{lemma}
\label{EGc} \cite{ErGa59}Let $k\geq2.$ If $G$ is a graph of order $n,$ with no
cycle longer than $k,$ then $e\left(  G\right)  \leq k\left(  n-1\right)  /2.$
\end{lemma}

Let $c(G)$ denote the circumference, i.e., the size of a longest cycle of $G.$
The following result is one case of Dirac theorem (see \cite{Dirac}).

\begin{lemma}
\label{Circumference} Let $G$ be a graph with $\delta\left(  G\right)  \geq2.$
Then $c(G)\geq\delta\left(  G\right)  +1$ holds.
\end{lemma}

To state the next result set $L_{t,k}:=K_{1}\vee tK_{k},$ i.e., $L_{t,k}$
consists of $t$ complete graphs of order $k+1,$ all sharing a single common
vertex. In \cite{AlSt96}, Ali and Staton gave the following stability result.

\begin{lemma}
\label{AS}Let $k\geq1,$ $n\geq2k+1,$ $G$ be a graph of order $n,$ and
$\delta\left(  G\right)  \geq k.$ If $G$ is connected, then $P_{2k+2}\subseteq
G,$ unless $G\subseteq S_{n,k},$ or $n=tk+1$ and $G=L_{t,k}$.
\end{lemma}

For the proof we also need the following two upper bounds on $q\left(
G\right)  .$

\begin{lemma}
\label{Merris bound} \cite{Mer98}For every graph $G,$ we have
\[
q\left(  G\right)  \leq\max_{u\in V\left(  G\right)  }\left\{  d\left(
u\right)  +\frac{1}{d\left(  u\right)  }\sum_{v\in\Gamma\left(  u\right)
}d\left(  v\right)  \right\}  .
\]

\end{lemma}

\begin{lemma}
\label{Das} \cite{Das04}If $G$ is a graph with $n$ vertices and $m$ edges,
then
\[
q\left(  G\right)  \leq\frac{2m}{n-1}+n-2.
\]

\end{lemma}

Now we will begin our analysis. In \cite{Nik201310} (see Proposition 2 of
\cite{Nik201310}), it was pointed out that when $k\geq2$ and $n>5k^{2},$
\begin{equation}
q\left(  S_{n,k}\right)  >n+2k-2-\frac{2k\left(  k-1\right)  }{n+2k-3}>n+2k-3.
\label{lb}%
\end{equation}
Then for a graph $G$ with $q(G)\geq q\left(  S_{n,k}\right)  ,$ we have
\[
n+2k-2-\frac{2k\left(  k-1\right)  }{n+2k-3}<q\left(  S_{n,k}\right)  \leq
q(G)\leq\frac{2e(G)}{n-1}+n-2,
\]
which implies
\[
2e(G)>2k\left(  n-1\right)  -2k\left(  k-1\right)  +\frac{4k\left(
k-1\right)  ^{2}}{n+2k-3}.
\]
Noting that $2e(G)$ is even, $2k\left(  n-1\right)  -2k\left(  k-1\right)  $
is also even, and $0<\frac{4k\left(  k-1\right)  ^{2}}{n+2k-3}<1,$ and then we
have
\[
2e(G)\geq2k\left(  n-1\right)  -2k\left(  k-1\right)  +2,
\]
i.e.,%
\begin{equation}
e(G)\geq kn-k^{2}+1. \label{edge of G}%
\end{equation}

Given a graph $G,$ write $G\left[  X\right]  $ for the graph induced by $X,$
when $X$ $\subseteq$ $V\left(  G\right)  $, and $e\left(  X\right)  $ for the
number of edges in graph $G\left[  X\right]  ,$ and $e\left(  X,Y\right)  $
for the number of edges joining vertices in $X$ to vertices in $Y,$ where $X$
and $Y$ are disjoint sets of $V\left(  G\right)  $. For a vertex $u\in
V\left(  G\right)  ,$ the following fact was pointed out in \cite{FNP13}.
\[
\sum_{v\in\Gamma\left(  u\right)  }d\left(  v\right)  =2e\left(  \Gamma\left(
u\right)  \right)  +e\left(  \Gamma\left(  u\right)  ,V\left(  G\right)
\backslash\Gamma\left(  u\right)  \right)  ,
\]
which is crucial in the proof of Lemma \ref{MxD}.

\begin{lemma}
\label{MxD}Let $k\geq2,$ $n\geq8k^{2},$ and let $G$ be a graph of order $n.$
If $G$ has no $C_{2k+1}$ and $q\left(  G\right)  \geq q\left(  S_{n,k}\right)
,$ then $\Delta(G)=n-1.$
\end{lemma}

\begin{proof}
Let $w$ be a vertex of $G$ such that
\[
d\left(  w\right)  +\frac{1}{d\left(  w\right)  }\sum_{i\in\Gamma\left(
w\right)  }d\left(  i\right)  =\max_{u\in V\left(  G\right)  }\left\{
d\left(  u\right)  +\frac{1}{d\left(  u\right)  }\sum_{v\in\Gamma\left(
u\right)  }d\left(  v\right)  \right\}  .
\]
Lemma \ref{Merris bound} implies that
\[
q\left(  G\right)  \leq d\left(  w\right)  +\frac{1}{d\left(  w\right)  }%
\sum_{i\in\Gamma\left(  w\right)  }d\left(  i\right)  .
\]
We shall show that\ if $d\left(  w\right)  <n-1,$ then $q\left(  G\right)
<q\left(  S_{n,k}\right)  .$

Set $A=\Gamma\left(  w\right)  ,$ $B=V\left(  G\right)  \backslash\left(
\Gamma\left(  w\right)  \cup\left\{  w\right\}  \right)  ;$ the assumption
$d\left(  w\right)  <n-1$ implies that $B\neq\varnothing.$ Note that
$P_{2k}\nsubseteq G\left[  A\right]  ;$ hence Lemma \ref{EGp} implies that
\begin{equation}
2e\left(  A\right)  \leq\left(  2k-2\right)  \left\vert A\right\vert .
\label{ub}%
\end{equation}

Our first goal is to show that $\left\vert B\right\vert <2k^{2}.$ Indeed, we
see that
\begin{equation}
q\left(  S_{n,k}\right)  \leq q\left(  G\right)  \leq d\left(  w\right)
+\frac{1}{d\left(  w\right)  }\sum_{i\in\Gamma\left(  w\right)  }d\left(
i\right)  =\left\vert A\right\vert +1+\frac{2e\left(  A\right)  +e\left(
A,B\right)  }{\left\vert A\right\vert }. \label{b}%
\end{equation}
If a vertex of $B$ is joined to every vertex of $A,$ then $G\left[  A\right]
$ does not contain a $P_{2k-1},$ as otherwise $C_{2k+1}\subset G.$ Hence,
Lemma \ref{EGp} implies that $2e\left(  A\right)  \leq\left(  2k-3\right)
\left\vert A\right\vert $ and from (\ref{b}) we see that
\[
q\left(  S_{n,k}\right)  \leq\left\vert A\right\vert +1+\frac{\left(
2k-3\right)  \left\vert A\right\vert +\left\vert A\right\vert \left\vert
B\right\vert }{\left\vert A\right\vert }=\left\vert A\right\vert +1+\left\vert
B\right\vert +2k-3=n+2k-3<q\left(  S_{n,k}\right)  .
\]
This contradiction shows that no vertex of $B$ is joined to all vertices of
$A;$ hence%
\[
e\left(  A,B\right)  \leq\left\vert B\right\vert \left(  \left\vert
A\right\vert -1\right)  .
\]
Now, inequalities (\ref{ub}) and (\ref{b}) imply that
\[
q\left(  S_{n,k}\right)  \leq\left\vert A\right\vert +1+\frac{\left(
2k-2\right)  \left\vert A\right\vert +\left\vert B\right\vert \left(
\left\vert A\right\vert -1\right)  }{\left\vert A\right\vert }=n+2k-2-\frac
{\left\vert B\right\vert }{n-1-\left\vert B\right\vert }.
\]
Comparing this inequality to (\ref{lb}), after some algebra we find that
$\left\vert B\right\vert <2k^{2},$ as claimed.

Next, let $A^{\prime}$ be the set of all vertices in $A$ that are joined to
each vertex in $B.$ Our next goal is to show that $\left\vert A^{\prime
}\right\vert \geq\left\vert A\right\vert -2k^{2}.$ Indeed, obviously
\[
e\left(  A,B\right)  \leq\left\vert A^{\prime}\right\vert \left\vert
B\right\vert +\left(  \left\vert A\right\vert -\left\vert A^{\prime
}\right\vert \right)  \left(  \left\vert B\right\vert -1\right)  =\left\vert
A\right\vert \left\vert B\right\vert -\left\vert A\right\vert +\left\vert
A^{\prime}\right\vert .
\]
Hence inequalities (\ref{ub}) and (\ref{b}) imply that
\[
q\left(  S_{n,k}\right)  \leq\left\vert A\right\vert +1+\frac{\left(
2k-2\right)  \left\vert A\right\vert +\left\vert A\right\vert \left\vert
B\right\vert -\left\vert A\right\vert +\left\vert A^{\prime}\right\vert
}{\left\vert A\right\vert }=n+2k-3+\frac{\left\vert A^{\prime}\right\vert
}{\left\vert A\right\vert }%
\]
Comparing this inequality to (\ref{lb}), after some algebra we find that
$\left\vert A^{\prime}\right\vert >\left\vert A\right\vert -2k^{2},$ as claimed.

Finally, let $G_{1}$ be the union of all components of $G\left[  A\right]  $
that contain a vertex from $A^{\prime},$ and let $G_{2}$ be the union of the
remaining components of $G\left[  A\right]  .$ As before we see that $e\left(
G_{2}\right)  \leq\left(  k-1\right)  v\left(  G_{2}\right)  .$

We claim that $G_{1}$ contains no cycle longer than $2k-3.$ Indeed, if
$C_{s}\subset G_{1}$ for some $s\geq2k-2,$ then $s\leq2k-1,$ as otherwise
$P_{2k}\subset G\left(  A\right)  ,$ contradicting that $C_{2k+1}\nsubseteq
G.$ The definition of $G_{1}$ implies that there must be a vertex in
$A^{\prime}$ that either belongs to $C_{s}$ or can be joined to a vertex of
$C_{s}$ by a path contained in $G_{1}.$ In either of these cases we see that
$G_{1}$ contains a path $P=P_{2k-2}$ with an endvertex belonging to
$A^{\prime}.$ Since
\[
\left\vert A^{\prime}\right\vert \geq\left\vert A\right\vert -2k^{2}%
=n-1-\left\vert B\right\vert -2k^{2}>n-1-4k^{2}>2k-2,
\]
we can extend $P$ by a vertex in $B$ and a vertex in $A^{\prime},$ thus
obtaining a path $P_{2k}\subset G\left(  \left[  A\cup B\right]  \right)  $
with two endvertices belonging to $A.$ Therefore $C_{2k+1}\subset G,$ a
contradiction showing that $G_{1}$ contains no cycle longer than $2k-3.$
Hence, Lemma \ref{EGc} implies that
\[
2e\left(  G_{1}\right)  \leq\left(  2k-3\right)  v\left(  G_{1}\right)  ,
\]
and so%
\begin{align*}
2e\left(  A\right)   &  =2e\left(  G_{1}\right)  +2e\left(  G_{2}\right)
\leq\left(  2k-3\right)  v\left(  G_{1}\right)  +\left(  2k-2\right)  v\left(
G_{2}\right) \\
&  =\left(  2k-3\right)  \left(  \left\vert A\right\vert -v\left(
G_{2}\right)  \right)  +\left(  2k-2\right)  v\left(  G_{2}\right)  =\left(
2k-3\right)  \left\vert A\right\vert +v\left(  G_{2}\right) \\
&  \leq\left(  2k-3\right)  \left\vert A\right\vert +\left\vert A\right\vert
-\left\vert A^{\prime}\right\vert <\left(  2k-3\right)  \left\vert
A\right\vert +2k^{2}.
\end{align*}
This inequality, together with (\ref{b}), implies that
\[
q\left(  S_{n,k}\right)  \leq\left\vert A\right\vert +1+\frac{\left(
2k-3\right)  \left\vert A\right\vert +2k^{2}+\left\vert A\right\vert
\left\vert B\right\vert }{\left\vert A\right\vert }=n+2k-3+\frac{2k^{2}%
}{\left\vert A\right\vert }.
\]
Comparing this bound with (\ref{lb}), we find that%
\[
1-\frac{2k\left(  k-1\right)  }{n+2k-3}<\frac{2k^{2}}{\left\vert A\right\vert
}\leq\frac{2k^{2}}{n-4k^{2}},
\]
which is a contradiction for $n\geq8k^{2}.$ Therefore $B=\varnothing$ and
$d\left(  w\right)  =n-1,$ completing the proof of Lemma \ref{MxD}.
\end{proof}

Let $w$ be a vertex of graph $G.$ Then we write $G_{w}=G[V(G)\backslash\{w\}]$
for short. We would like to point out that the main idea of the following
lemma comes from the Proof of Lemma 7 of \cite{Nik201310}.

\begin{lemma}
\label{Removing} Let $k\geq3,$ $n\geq110k^{2},$ and let $G$ be a graph of
order $n.$ If $G_{w}=G_{1}\cup G_{2},$ where $V\left(  G_{1}\right)  \ $and
$V\left(  G_{2}\right)  $ are disjoint, $P_{2k}\nsubseteq G_{1},$ and
$1\leq\nu\left(  G_{2}\right)  \leq k^{2}+k-3,$ then $q\left(  G\right)
<q\left(  S_{n,k}\right)  .$
\end{lemma}

\begin{proof}
Write $\nu\left(  G_{2}\right)  =t,$ then $1\leq t\leq k^{2}+k-3.$ Assume for
a contradiction that $q\left(  G\right)  \geq q\left(  S_{n,k}\right)  .$ We
may suppose that $w$ is a dominating vertex of $G$, and $G_{2}$ is isomorphic
to $K_{t},$ otherwise, we may add some edges to $G,$ while $q\left(  G\right)
$ will not decrease. Denote by $G_{0}$ the the graph obtained from $G$ by
removing $G_{2}.$ In view of $P_{2k}\nsubseteq G_{1},$ then by Lemma \ref{EGp}
we have
\[
e\left(  G_{1}\right)  \leq\left(  k-1\right)  \left(  n-t-1\right)  ,
\]
and then
\[
e\left(  G_{0}\right)  =e\left(  G_{1}\right)  +n-t-1\leq k\left(
n-t-1\right)  .
\]
Lemma \ref{Das} implies that
\[
q\left(  G_{0}\right)  \leq\frac{2k\left(  n-t-1\right)  }{n-t-1}%
+n-t-2=n+2k-t-2.
\]
Let $\left(  x_{1},\ldots,x_{n}\right)  ^{T}$ be a unit eigenvector to
$q\left(  G\right)  .$ By symmetry the entries corresponding to vertices of
$G_{2}$ have the same value, say $x$. From the eigenequations for $Q\left(
G\right)  $ we see that%

\[
\left(  q\left(  G\right)  -n+1\right)  x_{w}=\sum_{i\in V\left(  G\right)
\backslash\left\{  w\right\}  }x_{i}\leq\sqrt{\left(  n-1\right)  \left(
1-x_{w}^{2}\right)  },
\]
and noting that
\[
q\left(  G\right)  \geq q\left(  S_{n,k}\right)  >n+2k-3,
\]
so
\begin{equation}
x_{w}^{2}\leq\frac{n-1}{\left(  q\left(  G\right)  -n+1\right)  ^{2}+n-1}%
\leq\frac{n-1}{n-1+4\left(  k-1\right)  ^{2}}<1-\frac{4\left(  k-1\right)
^{2}}{n+4k^{2}}.\label{x_w}%
\end{equation}
Also from the eigenequations for $Q\left(  G\right)  $ we have
\[
x=\frac{x_{w}}{q\left(  G\right)  -2t+1}.
\]
Note that $q\left(  G\right)  >n+2k-3$ and $t\leq k^{2}+k-3,$ we have
\begin{equation}
x\leq\frac{x_{w}}{n+2k-2t-2}\leq\frac{x_{w}}{n-2k^{2}+4}.\label{x}%
\end{equation}
When $t\geq1,$ $n\geq110k^{2},$ by using (\ref{x_w}) and (\ref{x}) we find
that
\begin{align*}
q\left(  G\right)   &  =\sum_{ij\in E\left(  G\right)  }\left(  x_{i}%
+x_{j}\right)  ^{2}=\sum_{ij\in E\left(  G_{0}\right)  }\left(  x_{i}%
+x_{j}\right)  ^{2}+t\left(  x+x_{w}\right)  ^{2}+\sum_{ij\in E\left(
G_{2}\right)  }\left(  x_{i}+x_{j}\right)  ^{2}\\
&  \leq q\left(  G_{0}\right)  \left(  1-tx^{2}\right)  +t\left(
x+x_{w}\right)  ^{2}+2t\left(  t-1\right)  x^{2}\\
&  <n+2k-t-2+t\left(  1+\frac{1}{\left(  n-2k^{2}+4\right)  ^{2}}+\frac
{2}{n-2k^{2}+4}+\frac{2\left(  t-1\right)  }{\left(  n-2k^{2}+4\right)  ^{2}%
}\right)  x_{w}^{2}\\
&  <n+2k-t-2+t\left(  1+\frac{3}{n-2k^{2}+4}\right)  \left(  1-\frac{4\left(
k-1\right)  ^{2}}{n+4k^{2}}\right)  \\
&  <n+2k-2-\left(  \frac{4\left(  k-1\right)  ^{2}}{n+4k^{2}}-\frac
{3}{n-2k^{2}+4}\right)  \\
&  <n+2k-2-\frac{2k\left(  k-1\right)  }{n-2k^{2}+4}\\
&  <q(S_{n,k}).
\end{align*}

Therefore $q\left(  G\right)  <q(S_{n,k}),$ and this contradiction completes
the proof.
\end{proof}

We will call vertex $v$ a \emph{center vertex} of graph $S_{n,k},$ if
$d\left(  v\right)  =n-1$ holds.

\begin{lemma}
\label{one S} Let $G$ be a graph of order $n.$ If $G_{w}=\cup_{i=1}%
^{t}S_{n_{i},k-1},k\geq3,$ and $t\geq2$ hold, then we have $q\left(  G\right)
<q\left(  S_{n,k}\right)  .$
\end{lemma}

\begin{proof}
We may suppose $w$ is a dominating vertex of $G,$ otherwise we may add some
edges to $G,$ and $q(G)$ will not decrease. We first consider $t=2,$ that is
to say,
\[
G_{w}=S_{n_{1},k-1}\cup S_{n_{2},k-1}.
\]
Let $u_{1},\cdot\cdot\cdot,u_{k-1}\ $be all the center vertices of
$S_{n_{1},k-1},$ and $v_{1},\cdot\cdot\cdot,v_{k-1}\ $be all the center
vertices of $S_{n_{2},k-1}.$ Let $\mathbf{x}$ be a unit eigenvector to
$q\left(  G\right)  .$ Then by symmetry we have $x_{u_{1}}=\cdot\cdot
\cdot=x_{u_{k-1}},$ and $x_{v_{1}}=\cdot\cdot\cdot=x_{v_{k-1}}.$ Without loss
of generality we assume that $x_{u_{1}}\geq x_{v_{1}}.$

Now combine the components $S_{n_{1},k-1}$ and $S_{n_{2},k-1}$ into
$S_{n_{1}+n_{2},\text{ }k-1},$ and let $u_{1},\cdot\cdot\cdot,u_{k-1}\ $be the
center vertices of $S_{n_{1}+n_{2},\text{ }k-1}.$ Denote by $G^{\prime}$ be
the graph obtained from $G$ by the above perturbation. Set
\[
W=V\left(  S_{n_{2},\text{ }k-1}\right)  \backslash\left\{  v_{1},\cdot
\cdot\cdot,v_{k-1}\right\}  .
\]
Then
\begin{align*}
q\left(  G^{\prime}\right)  -q\left(  G\right)   &  \geq\mathbf{x}^{T}Q\left(
G^{\prime}\right)  \mathbf{x}-\mathbf{x}^{T}Q\left(  G\right)  \mathbf{x}\\
&  =\sum_{i\in W}\left(  k-1\right)  \left[  \left(  x_{u_{1}}+x_{i}\right)
^{2}-\left(  x_{v_{1}}+x_{i}\right)  ^{2}\right]  +\left(  k-1\right)
^{2}\left(  x_{u_{1}}+x_{v_{1}}\right)  ^{2}-2\left(  k-1\right)  \left(
k-2\right)  x_{v_{1}}^{2}\\
&  >0.
\end{align*}
Noting that $G^{\prime}=S_{n,k},$ then we have
\[
q\left(  G\right)  <q\left(  G^{\prime}\right)  =q\left(  S_{n,k}\right)  .
\]
When $t\geq3,$ we may prove the lemma by using induction on $t$ and applying
the above perturbation to $G$ repeatedly$.$
\end{proof}

\begin{lemma}
\label{subgaph H} Let $G$ be a graph of order $n$ with $e\left(  G\right)
\geq\left(  k-1\right)  n-\left(  k^{2}-k-1\right)  .$ If $P_{2k}\nsubseteq$
$G,$ then there exists an induced subgraph $H$ of $G$ such that $\delta(H)\geq
k-1$ and $\nu\left(  H\right)  \geq n-\left(  k^{2}-k-1\right)  .$
\end{lemma}

\begin{proof}
Define a sequence of graphs, $G_{0}\supset G_{1}\supset\cdots\supset G_{r}$
using the following procedure$.$

$G_{0}:=G;$

$i:=0$;

\textbf{while} $\delta(G_{i})<k-1$ \textbf{do begin}

\qquad select a vertex $v\in V(G_{i})$ with $d(v)=\delta(G_{i});$

\qquad$G_{i+1}:=G_{i}-v;$

$\qquad i:=i+1;$

\textbf{end.}

Note that the whole loop must exit before $i=k^{2}-k$. Indeed, $P_{2k}%
\nsubseteq G_{i},$ Lemma \ref{EGp} implies that
\begin{equation}
e(G_{i})\leq\left(  k-1\right)  \left(  n-i\right)  . \label{1}%
\end{equation}

On the other hand,
\begin{equation}
e(G_{i})\geq e\left(  G\right)  -i\left(  k-2\right)  \geq\left(  k-1\right)
n-\left(  k^{2}-k-1\right)  -i\left(  k-2\right)  . \label{2}%
\end{equation}

Then from (\ref{1})  and (\ref{2}) we have $i\leq k^{2}-k-1.$ Let $H=G_{r},$
where $r$ is the last value of the variable $i,$ and then the proof is completed.
\end{proof}

\section{\label{pf} Proof of Theorem \ref{main result}}

\begin{proof}
[Proof of Theorem \ref{main result}]Assume for a contradiction that $q\left(
G\right)  \geq q\left(  S_{n,k}\right)  $. By virtue of Lemma \ref{MxD} we
suppose that $w$ is a dominating vertex of $G.$ Then from (\ref{edge of G}),
we have
\[
e(G_{w})\geq\left(  k-1\right)  \left(  n-1\right)  -\left(  k^{2}-k-1\right)
.
\]
By taking $G_{w}$ as the graph $G$ in Lemma \ref{subgaph H}, we may obtain an
induced subgraph $H$ of $G_{w}$ such that $\delta(H)\geq k-1$ and $\nu
(H)\geq\left(  n-1\right)  -\left(  k^{2}-k-1\right)  .$ Set%
\[
H=\cup_{i=1}^{t}H_{i}\text{, and }\nu(H_{i})=h_{i},
\]
where each $H_{i}$ is a component of $H.$ By virtue of Dirac theorem (see
Lemma \ref{Circumference}), $\delta(H_{i})\geq k-1\geq2$ implies that
$C_{l}\subseteq H_{i}$ for some $l\geq k.$ Then any component of $G_{w}$
contains at most one graph of $\left\{  H_{1,}H_{2,}\cdot\cdot\cdot
,H_{t}\right\}  $ as an induced subgraph, otherwise $P_{2k+1}\subseteq$
$G_{w}$ and then $C_{2k+1}\subseteq G.$ Now for each $1\leq i\leq t,$ let
$F_{i}$ be the component of $G_{w},$ which contains $H_{i}$ as an induced
subgraph, and set%
\[
G_{w}=\left(  \cup_{i=1}^{t}F_{i}\right)  \cup F_{0}.
\]
Obviously, $P_{2k}\nsubseteq F_{i}$ for any $0\leq i\leq t.$

We claim that $F_{0}=\emptyset,$ otherwise
\[
1\leq\nu\left(  F_{0}\right)  \leq|V\left(  G_{w}\right)  \backslash V\left(
H\right)  |\leq k^{2}-k-1,
\]
and then by Lemma \ref{Removing} we obtain the contradiction $q\left(
G\right)  <q\left(  S_{n,k}\right)  .$ By the same token we claim that
$h_{i}\geq2k-1$ for each $1\leq i\leq t.$ Otherwise the order of the component
$F_{i}$ satisfies
\[
k\leq h_{i}\leq\nu\left(  F_{i}\right)  \leq h_{i}+|V\left(  G_{w}\right)
\backslash V\left(  H\right)  |\leq2k-2+k^{2}-k-1=k^{2}+k-3.
\]
Then by virtue of Lemma \ref{Removing}, we also obtain the contradiction
$q\left(  G\right)  <q\left(  S_{n,k}\right)  .$

Since $P_{2k}\nsubseteq H_{i}$ and $\delta(H_{i})\geq k-1,$ from Lemma
\ref{AS}, we deduce two cases for the structure of each $H_{i}$.

(a) $H_{i}\subseteq$ $S_{h_{i},k-1},$ and then
\[
e\left(  H_{i}\right)  \leq\left(  k-1\right)  h_{i}-\frac{k\left(
k-1\right)  }{2}.
\]

(b) $H_{i}=$ $L_{p_{i},k-1}$ with $h_{i}=p_{i}\left(  k-1\right)  +1,$ and
then
\[
e\left(  H_{i}\right)  =\frac{k\left(  h_{i}-1\right)  }{2}<\left(
k-1\right)  h_{i}-\frac{k\left(  k-1\right)  }{2}.
\]
We claim that $H_{i}\neq$ $L_{p_{i},k-1}$ with $h_{i}=p_{i}\left(  k-1\right)
+1.$ Assume for a contradiction that some $H_{i}$ is isomorphic to
$L_{p_{i},k-1}$ with $h_{i}=p_{i}\left(  k-1\right)  +1,$ then
\[
e\left(  H\right)  =%
{\displaystyle\sum\limits_{j=1}^{t}}
e\left(  H_{j}\right)  \leq\left(  k-1\right)  \left(  \nu(H)-h_{i}\right)
-\frac{k\left(  k-1\right)  }{2}+\frac{k\left(  h_{i}-1\right)  }{2}.
\]
On the other hand, from the procedure of Lemma \ref{subgaph H}, we know that
\[
e\left(  H\right)  \geq e\left(  G_{w}\right)  -\left(  n-1-\nu(H)\right)
\left(  k-2\right)  \geq\left(  k-1\right)  \left(  n-1\right)  -\left(
k^{2}-k-1\right)  -\left(  n-1-\nu(H)\right)  \left(  k-2\right)  .
\]
Therefore%
\[
\left(  k-1\right)  \left(  n-1\right)  -\left(  k^{2}-k-1\right)  -\left(
n-1-\nu(H)\right)  \left(  k-2\right)  \leq\left(  k-1\right)  \left(
\nu(H)-h_{i}\right)  -\frac{k\left(  k-1\right)  }{2}+\frac{k\left(
h_{i}-1\right)  }{2},
\]
which implies that $h_{i}<k+1,$ and this is a contradiction to $h_{i}\geq
2k-1$. So each $H_{i}$ is a subgraph of $S_{h_{i},k-1}$.

For fixed $i,$ assume now that $I$ is the independent set of $H_{i}$ of order
$h_{i}-\left(  k-1\right)  ,$ and set $J=V\left(  H_{i}\right)  \backslash I.$
Clearly, $\delta\left(  H_{i}\right)  \geq k-1$ implies that every vertex of
$I$ is joined to every vertex in $J.$ Therefore, for any two vertices in $I$
there exists a path of order $2k-1$ with them as endvertices$.$ If $u$ is a
vertex in $V\left(  F_{i}\right)  \backslash V\left(  H_{i}\right)  $ and
$\Gamma\left(  u\right)  \cap V\left(  H_{i}\right)  \neq\emptyset,$ then we
have $\Gamma_{F_{i}}\left(  u\right)  \subseteq J,$ otherwise $P_{2k}\subseteq
F_{i}.$ Furthermore, for any vertex $v\in V\left(  F_{i}\right)  \backslash
V\left(  H_{i}\right)  $ we may deduce that $\Gamma\left(  v\right)  \cap
V\left(  H_{i}\right)  \neq\emptyset$, hence $\Gamma_{F_{i}}\left(  v\right)
\subseteq J$. Therefore $\left(  V\left(  F_{i}\right)  \backslash V\left(
H_{i}\right)  \right)  \cup I$ is an independent set of $F_{i}$, and then
$F_{i}$ is a subgraph of $S_{\nu\left(  F_{i}\right)  ,k-1}.$ Thus
\[
G_{w}=\cup_{i=1}^{t}F_{i},
\]
where each $F_{i}$ is a subgraph of $S_{\nu\left(  F_{i}\right)  ,k-1}.$ Note
that $q\left(  G\right)  $ will not decrease when adding some edges to $G,$
hence we may suppose that $F_{i}$ is isomorphic to $S_{\nu\left(
F_{i}\right)  ,k-1}.$ If $t\geq2,$ then by Lemma \ref{one S} we deduce the
contradiction $q\left(  G\right)  <q\left(  S_{n,k}\right)  .$ If $t=1,$ we
have $q\left(  G\right)  \leq q\left(  S_{n,k}\right)  $ with equality holding
if and only if $G_{w}=S_{n-1,k-1}$ and then $G=S_{n,k}.$
\end{proof}

\bigskip

\textbf{Acknowledgement}

The author would like to thank Prof. Vladimir Nikiforov for his kind
encouragement and invaluable suggestions, and thank the referees for their
careful reading and useful comments.

\end{document}